\documentclass[10pt]{article}

\usepackage[hmargin=35mm,vmargin=40mm]{geometry}
\usepackage{amssymb}
\usepackage{amsmath}
\usepackage{algorithm}
\usepackage{amsthm}
\usepackage{graphicx}
\usepackage{subfigure}
\usepackage{natbib}
\usepackage[noend]{algpseudocode} 

\begin{document}

\title{Trace Norm Regularized Tensor Classification and Its Online Learning Approaches}

\author{Ziqiang Shi, Tieran Zheng, and Jiqing Han}

\maketitle

\begin{abstract}
In this paper we propose an algorithm to classify tensor data. Our methodology is built on recent studies about matrix classification with the trace norm constrained weight matrix and the tensor trace norm. Similar to matrix classification, the tensor classification is formulated as a convex optimization problem which can be solved by using the off-the-shelf accelerated proximal gradient (APG) method. However, there are no analytic solutions as the matrix case for the updating of the weight tensors via the proximal gradient. To tackle this problem, the Douglas-Rachford splitting technique and the alternating direction method of multipliers (ADM) used in tensor completion are adapted to update the weight tensors. Further more, due to the demand of real applications, we also propose its online learning approaches. Experiments demonstrate the efficiency of the methods.
\end{abstract}

\section{Introduction}
Tensor or multi-way data analysis have many applications in the field of psychometrics, econometrics, image processing, signal precessing, neuroscience, and data mining [1]. Tensors are higher-order equivalent of vectors and matrices. In this paper, we consider the classification of tensors, which is a generalization of the matrices classification problem proposed by Tomioka and Aihara in [2]. The tensor classification model is formulated as:
\begin{equation}\label{eq:TensorLinearRegression}
f(\mathcal{X};\mathcal{W},b)=<\mathcal{W},\mathcal{X}>+b
\end{equation}
where $\mathcal{W},\mathcal{X}\in \mathbb{R}^{I_1\times I_2\times \cdots \times I_N}$ are $N$-way tensors, $\mathcal{X}$ is the input tensor for which we would like to predict its class label $y$; $\mathcal{W}$ is called the \emph{weight tensor} and $b\in \mathbb{R}$ is the \emph{bias}. Thus we need to infer the \emph{weight tensor} and \emph{bias} from the training samples $\{\mathcal{X}_i,y_i\}^{s}_{i=1}$. This formulation makes the work in [2] as a special case that the tensors evolved have an order of $N=2$.

In the work of matrix classification, Tomioka and Aihara use a norm regularized scheme based on trace norm of the weight matrix [2]. Recently this trace norm regularization scheme has been studied in various contexts, namely, multi-task learning [3], matrix completion [4,5],  and robust principle component analysis [6]. In this paper, similarly to matrix classification, a trace norm for tensors may be introduced to control the complexity of the weight tensor and the deviation of the empirical statistics from the predictions together. Recently, Liu et al. [7] proposed a definition for the tensor trance norm:
\begin{equation}\label{eq:TensorTraceNormDefinition}
\|\mathcal{X}\|_*:=\frac{1}{N} \sum\limits_{i=1}^N \|X_{(i)}\|_*
\end{equation}
where $X_{(i)}$ is the mode-$i$ unfolding of $\mathcal{X}$, $\|X_{(i)}\|_*$ is the trace norm of the matrix $X_{(i)}$, i.e. the sum of the singular values of $X_{(i)}$, and if $N=2$, this tensor norm is just the ordinary matrix trace norm. Now the weight tensor and bias learning problem becomes a convex optimization problem
\begin{equation}\label{eq:OptimizationProblemFormulation}
\mathop {\min}\limits_{\mathcal{W},b}F_s (\mathcal{W},b)=f_s(\mathcal{W},b)+\lambda\left\|\mathcal{W}\right\|_*,
\end{equation}
where $f_s(\mathcal{W},b)=\sum\nolimits_{i=1}^s {\ell(y_i, <\mathcal{W},\mathcal{X}_i>+b)}$ is the empirical cost function induced by some convex smooth loss function $\ell(\cdot,\cdot)$, and $\lambda$ is the regularization parameter. The subscript of $f_s(W,b)$ indicates the number of training samples or time of training procedure which is apparent from context.

For such convex optimization problem, Toh and Yun [8], Ji and Ye [9], and Liu et al. [10] independently proposed similar algorithms in the context of matrix related problems via using accelerated proximal gradient (APG) based methods. In this paper, we adapted the APG based algorithm to this tensor convex optimization problem. Unfortunately, unlike the Theorem 3.1 in [9] for matrix case, there is no closed analytic solution of the weight updating rules in the APG algorithm for the tensor case due to the dependency among multiple constraints. In order to solve the weight updating problem, the Douglas-Rachford splitting technique and the alternating direction method of multipliers [15,16], which have been successfully used in tensor completion tasks [7,11], are employed.

Furthermore, in order to cope with the situations that huge size training set for the data cannot be loaded into the memory simultaneously or the training data appear in sequence (for example video processing), we propose the online implementations of the above algorithms.

\section{Notations}
\label{sec:Notations}

We adopt the nomenclature used by Kolda and Bader on tensor decompositions and applications [1]. The \emph{order} $N$ of a tensor is the number of dimensions, also known as ways or modes. Matrices (tensor of order two) are denoted by upper case letters, e.g. $X$, and lower case letters for the elements, e.g. $x_{ij}$. Higher-order tensors (order three or higher) are denoted by Euler script letters, e.g. $\mathcal{X}$, and element $(i_1,i_2,\cdots,i_N)$ of a $N$-order tensor $\mathcal{X}$ is denoted by $x_{i_1i_2\cdots i_N}$. \emph{Fibers} are the higher-order analogue of matrix rows and columns. A fiber is defined by fixing every index but one. The mode-$n$ fibers are all vectors $x_{i_1\cdots i_{n-1}:i_{n+1}\cdots i_N}$ that obtained by fixing the values of $\{i_1, i_2,\cdots, i_N\} \setminus i_n$. The mode-$n$ \emph{unfolding}, also knows as \emph{matricization}, of a tensor $\mathcal{X}\in \mathbb{R}^{I_1\times I_2\times \cdots \times I_N}$ is denoted by $X_{(n)}$ and arranges the model-$n$ fibers to be the columns of the resulting matrix. The unfolding operator is denoted as $\text{unfold}(\cdot)$. The opposite operation is $\text{refold}(\cdot)$, denotes the refolding of the matrix into a tensor.  The tensor element $(i_1,i_2,\cdots,i_N)$ is mapped to the matrix element $(i_n,j)$, where
\begin{equation*}
j=1+\displaystyle{\sum\limits_{\begin{subarray}{|}k=1\\ k\neq n\end{subarray}}^N (i_k-1)J_k}\quad  \textrm{ with  }\quad   J_k=\prod\limits_{\begin{subarray}{|}m=1 \\ m\neq n\end{subarray} } ^{k-1} I_m
\end{equation*}
Therefore, $X_{(n)}\in \mathbb{R}^{I_n\times I_1\cdots I_{n-1}I_{n+1}\cdots I_N}$. The $n$-\emph{rank} of a $N$-dimensional tensor $\mathcal{X}$, denoted as $\text{rank}_n(\mathcal{X})$ is the column rank of $X_{(n)}$, i.e. the dimension of the vector space spanned by the mode-$n$ fibers. The inner product of two same-size tensors $\mathcal{X},\mathcal{Y}\in \mathbb{R}^{I_1\times I_2\times \cdots \times I_N}$ is defined as
\begin{equation*}
<\mathcal{X},\mathcal{Y}>=\sum\limits_{i_1=1}^{I_1} \sum\limits_{i_2=1}^{I_2}\cdots \sum\limits_{i_N=1}^{I_N}x_{i_1i_2\cdots i_N}y_{i_1i_2\cdots i_N}.
\end{equation*}
The corresponding norm is $\|\mathcal{X}\|_F=\sqrt{<\mathcal{X},\mathcal{X}>}$, which is often called the Frobenius norm.


\section{Accelerated Proximal Gradient Method}
It is known [8] that the gradient step
\begin{equation}
\mathcal{W}_{k}=\mathcal{W}_{k-1}-\frac{1}{t_k}\nabla_\mathcal{W} f_s(\mathcal{W}_{k-1},b)
\end{equation}
for solving the following smooth problem with fixed bias $b$
\begin{equation}\label{eq:OptimizationProblemFormulationWithoutTraceNorm}
\mathop {\min}\limits_{\mathcal{W}}f_s(\mathcal{W},b)
\end{equation}
without trace norm regularization can be formulated equivalently as a proximal regularization of the linearized function $f_s(\mathcal{W},b)$ at $\mathcal{W}_{k-1}$ as
\begin{equation}\label{eq:ProximalRegularization}
\mathcal{W}_{k}=\mbox{arg}\!\mathop {\min}\limits_{\mathcal{W}}P_{t_k}(\mathcal{W},\mathcal{W}_{k-1}),
\end{equation}
where
\begin{equation}\label{eq:LinearizedFunction}
P_{t_k}(\mathcal{W},\mathcal{W}_{k-1})=f_s(\mathcal{W}_{k-1},b)+<\mathcal{W}-\mathcal{W}_{k-1},\nabla_\mathcal{W} f_s(\mathcal{W}_{k-1},b)>+\frac{t_k}{2}\|\mathcal{W}-\mathcal{W}_{k-1}\|_F^2
\end{equation}
and $\nabla_\mathcal{W} f_s(\cdot,b)$ is the gradient of $f_s(\cdot,b)$ with respect to $\mathcal{W}$.

Based on this equivalence relation, Toh and Yun [8], Ji and Ye [9], and Liu et al. [10] proposed to solve the optimization problem in Eq.~(\ref{eq:OptimizationProblemFormulation}) by the following iterative step:
\begin{equation}\label{eq:APGiteration}
\mathcal{W}_{k}=\mbox{arg}\!\mathop {\min}\limits_{\mathcal{W}}Q_{t_k}(\mathcal{W},\mathcal{W}_{k-1})\triangleq P_{t_k}(\mathcal{W},\mathcal{W}_{k-1})+\lambda \|\mathcal{W}\|_*
\end{equation}
or equivalently
\begin{equation}\label{eq:APGiterationSimpleForm}
\mathcal{W}_{k}=\mbox{arg}\!\mathop {\min}\limits_{\mathcal{W}}\{\frac{t_k}{2}\|\mathcal{W}-(\mathcal{W}_{k-1}-\frac{1}{t_k}\nabla_\mathcal{W} f_s(\mathcal{W}_{k-1},b))\|_F^2+\lambda \|\mathcal{W}\|_*\}.
\end{equation}

Unfortunately, when the order of the tensor evolved in the problem is three or higher, there is no closed analytic solution to the above problem due to the tensor norm. This is contrast to the matrix case, where the Eq.~(\ref{eq:APGiterationSimpleForm}) can be solved by singular value decomposition (SVD) and soft ``shrinkage'' like the theorem 3.1 in [9]. However, the Douglas-Rachford splitting technique and the alternating direction method of multipliers can be used to solve Eq.~(\ref{eq:APGiterationSimpleForm}) for higher tensors. These methods will be described in the next section. Now, we assume that the Eq.~(\ref{eq:APGiteration}) or Eq.~(\ref{eq:APGiterationSimpleForm}) can be properly solved.

In general APG methods, the Lipschitz constant for $\nabla_\mathcal{W} f_s(\cdot,b)$ is unknown, so it is need to estimate the appropriate step size $t_k$ to guarantee the convergence rate [8,9,10]. In this work, the standard squared loss function is used in Eq.~(\ref{eq:OptimizationProblemFormulation}). With this loss function, we can explicitly compute the Lipschitz constant in Lemma~\ref{lemma:LipschitzConstant}. Thus the step size estimation can be omitted in our tensor classification problems.

\newtheorem{thm}{Theorem}[section]
\newtheorem{cor}[thm]{Corollary}
\newtheorem{lem}[thm]{Lemma}

\begin{lem}
\label{lemma:LipschitzConstant}
$\nabla_W f_s(\cdot,b)$ is Lipschitz continuous with constant $L=2\prod\nolimits_{m=1}^{N} I_m\sum\limits_{i=1}^s \left\|\mathcal{X}_i\right\|_F^2$, i.e.,
\begin{equation}\label{eq:LipschitzCondition}
\left\|\nabla_\mathcal{W} f_s(\mathcal{U},b)-\nabla_\mathcal{W} f_s(\mathcal{V},b)\right\|_F\leq L\left\|\mathcal{U}-\mathcal{V}\right\|_F, \forall \mathcal{U}, \mathcal{V}\in \mathbb{R}^{I_1\times I_2\times \cdots \times I_N},
\end{equation}
where $\left\|\cdot\right\|_F$ denotes the Frobenius norm.
\end{lem}

\begin{proof}
With the standard squared loss, the gradient of $f_s(\mathcal{W},b)$ with respect to $\mathcal{W}$ is
\begin{equation}\label{eq:Gradient}
\nabla_\mathcal{W} f_s(\mathcal{W},b)=-2\sum\limits_{i=1}^s {(y_i-<\mathcal{W},\mathcal{X}_i>-b)\mathcal{X}_i},
\end{equation}
Applying Eq.~(\ref{eq:Gradient}) with $U,\mathcal{V}$ to the right of Eq.~(\ref{eq:LipschitzCondition}),
we obtain
\begin{eqnarray*}
& &\left\|\nabla_\mathcal{W} f_s(\mathcal{U},b)-\nabla_\mathcal{W} f_s(\mathcal{V},b)\right\|_F \\
&=&\left\| -2\sum\nolimits_{i=1}^s {(y_i-<\mathcal{U},\mathcal{X}_i>-b)\mathcal{X}_i} + 2\sum\nolimits_{i=1}^s {(y_i-<\mathcal{V},\mathcal{X}_i>-b)\mathcal{X}_i} \right\|_F \\
&=&2\left\| \sum\nolimits_{i=1}^s {(<\mathcal{U},\mathcal{X}_i>-<\mathcal{V},\mathcal{X}_i>)\mathcal{X}_i}\right\|_F \\
&\le&2\sum\nolimits_{i=1}^s \left| <\mathcal{U}-\mathcal{V},\mathcal{X}_i>\right| \left\| \mathcal{X}_i\right\|_F \\
&\le&2\prod\limits_{m=1}^{N} I_m\sum\nolimits_{i=1}^s \left\| \mathcal{U}-\mathcal{V}\right\|_F \left\| \mathcal{X}_i\right\|_F^2 \\
&=&(2\prod\limits_{m=1}^{N} I_m\sum\nolimits_{i=1}^s  \left\| \mathcal{X}_i\right\|_F^2)\left\| \mathcal{U}-\mathcal{V}\right\|_F,
\end{eqnarray*}
where in the last inequality, the easily verified fact that $<\mathcal{A},\mathcal{B}>\le \left\| \mathcal{A}\right\|_1\left\| \mathcal{B} \right\|_1 \le \prod\nolimits_{m=1}^{N} I_m\left\| \mathcal{A}\right\|_F\left\| \mathcal{B} \right\|_F$ for
$\forall \mathcal{A},\mathcal{B} \in \mathbb{R}^{I_1\times I_2\times \cdots \times I_N}$ is used. Here $\left\| \cdot \right\|_1$ denotes the $\ell_1$ norm which is the sum of the absolute values of the tensor elements.

Thus the lemma is proved, that is to say $\nabla_\mathcal{W} f_s(\cdot,b)$ is Lipschitz continuous with constant $L=2\prod\nolimits_{m=1}^{N} I_m\sum\nolimits_{i=1}^s \left\|\mathcal{X}_i\right\|_F^2$.
\end{proof}

Based on the the work of Nesterov [13,14], Toh and Yun [8], Ji and Ye [9], and Liu et al. [10] showed that introduce a search point sequence $\mathcal{Z}_k=\mathcal{W}_k+\frac{t_{k-1}-1}{t_k}(\mathcal{W}_k-\mathcal{W}_{k-1})$ for a sequence ${t_k}$ satisfying $t_{k+1}^2-t_{k+1}\leq t_{k}^2$ results in a convergence rate of $O(\frac{1}{k_2})$. Based on their results, we adapted the APG algorithm to the tensor classification case and summarized in Algorithm~\ref{algo:APG}. In this algorithm, the step of line 2 is not explicit. In the next section we will introduce some methods to solve this problem.

When the weight tensor is obtained, the bias $b$ can be derived by solving the following problem with fixed weight tensor
\begin{equation}\label{eq:BiasUpdateProblem}
b_k = \mbox{arg}\!\mathop {\min}\limits_{b}\{\sum\limits_{i=1}^s {(y_i-<\mathcal{W}_k,\mathcal{X}_i>-b)^2}+\lambda\left\|\mathcal{W}_k\right\|_*\},
\end{equation}
which results in the bias updating rule
\begin{equation}\label{eq:BiasUpdateRule}
b_k = \frac{1}{s}\sum\limits_{i=1}^s (y_i-<\mathcal{W}_k,\mathcal{X}_i>).
\end{equation}

\begin{algorithm}[t]
\caption{Weight Tensor Learning via APG} \label{algo:APG}
\textbf{Input} $(\mathcal{X}_i,y_i),i=1,\cdots,s.$

\textbf{Initialization} $\mathcal{W}_0=\mathcal{Z}_1\in  \mathbb{R}^{I_1\times I_2\times \cdots \times I_N}, \alpha_1 =1, L=2\prod\nolimits_{m=1}^{N} I_m\sum\limits_{i=1}^s \left\|\mathcal{X}_i\right\|_F^2, \lambda, k=1. $

1: \textbf{while} not converged \textbf{do}

2: $\mathcal{W}_k=\mbox{arg}\!\mathop {\min}\limits_{\mathcal{W}}\{\frac{L}{2}\|\mathcal{W}-(\mathcal{Z}_{k}-\frac{1}{L}\nabla_\mathcal{W} f_s(\mathcal{Z}_{k},b))\|_F^2+\lambda \|\mathcal{W}\|_*\}$.

3: $\alpha_{k+1}=\frac{1+\sqrt{1+4\alpha_k^2}}{2}$.

4: $\mathcal{Z}_{k+1}=\mathcal{W}_k+\frac{\alpha_{k}-1}{\alpha_{k+1}}(\mathcal{W}_k-\mathcal{W}_{k-1})$.

5: $k\leftarrow k+1$.

6: \textbf{end while}

\textbf{Output}: $\mathcal{W}\leftarrow \mathcal{W}_k$.
\end{algorithm}




\section{Minimization via Gandy's Algorithms}
\label{Gandy_Algorithm}
Apparently that the problem of Eq.~(\ref{eq:APGiterationSimpleForm}) or line 2 in Algorithm~\ref{algo:APG} fulfils the recently proposed tensor completion formulation [7,11].
For tensor completion, Gandy proposed two algorithms based on Douglas-Rachford splitting technique and the alternating direction method of multipliers (ADM) respectively. In this work, we adapt these two methods to solve the problem~(\ref{eq:APGiterationSimpleForm}).

\emph{Douglas-Rachford splitting technique based method}: The Douglas-Rachford splitting technique has a long history [15,16]. It addresses the minimization of the sum of two functions $(f+g)(x)$, where $f$ and $g$ are lower semicontinuous convex functions. The Douglas-Rachford splitting technique asserted that $\text{prox}_{\lambda g}(\tilde{x})$ is a minimizer of $(f+g)(x)$, where $\tilde{x}$ is the limit point of the following sequence:
\begin{equation}\label{eq:ConvergenceSequence}
x_{n+1}:=x_n+t_n\{\text{prox}_{\lambda f}[2\text{prox}_{\lambda g}(x_n)-x_n]-\text{prox}_{\lambda g}(x_n)\},
\end{equation}
where $t_n\in [0,2]$ satisfies $\sum \nolimits_{n\geq 0} t_n(2-t_n)=\infty$ and the proximal map $\text{prox}_{\lambda g}(\cdot)$ is defined as [17,18]:
\begin{equation}\label{eq:ProximalMapDefinition}
\text{prox}_{\lambda f}: x \mapsto \mbox{arg}\!\mathop {\min}\limits_{y}\{f(y)+\frac{1}{2\lambda}\|x-y\|^2 \}.
\end{equation}
We first formulate the problem in step 2 of Algorithm~\ref{algo:APG} into the unconstrained minimization of $(f+g)(x)$. Let $\mathfrak{F}:=\mathbb{R}^{I_1\times I_2\times \cdots \times I_N}$, define a Hilbert space $\mathfrak{H}_0:=\underbrace{\mathfrak{F} \times \mathfrak{F} \times \cdots \times  \mathfrak{F}}_{N+1 \text{ terms}}$ with the inner product $<\mathfrak{X},\mathfrak{Y}>_{\mathfrak{H}_0}:=\frac{1}{N+1}\sum\nolimits_{i=0}^N <\mathcal{X}_i,\mathcal{Y}_i>$. Then the problem can be rephrased as:
\begin{equation}\label{eq:DRProblemFormulation}
\mathop {\text{minimize}}\limits_{\mathfrak{W} \in \mathfrak{H}_0}\quad  f(\mathfrak{W})+g(\mathfrak{W}),
\end{equation}
where $\mathfrak{W}=(\mathcal{W}_0,\mathcal{W}_1,\cdots,\mathcal{W}_N)$, $D=\{\mathfrak{W}\in \mathfrak{H}_0|\mathcal{W}_0=\mathcal{W}_1=\cdots=\mathcal{W}_N\}$, and
\begin{eqnarray}\label{eq:DRFGfunction}
f(\mathfrak{W})=\frac{L}{2}\|\mathcal{W}_{0}-\mathcal{P}\|_F^2+\sum\limits_{i=1}^{N}\frac{\lambda}{N}\|W_{i,(i)}\|_*,\\
g(\mathfrak{W})=i_D(\mathfrak{W})=\left\{
\begin{aligned}
\begin{array}{l}
\end{array}
0,\text{ if }\mathfrak{W}\in D \\
+\infty, \text{ otherwise }
\end{aligned}
\right.
\end{eqnarray}
where $\mathcal{P}=\mathcal{Z}_{k-1}-\frac{1}{L}\nabla_\mathcal{W} f_s(\mathcal{Z}_{k-1},b)$. Then in order to apply the stand DR splitting technique, the proximal maps of $f(\mathfrak{W})$ and $g(\mathfrak{W})$ need to be identified.

The proximal map of $f(\mathfrak{W})$ is given by
\begin{eqnarray*}\label{eq:FproximalMap}
\text{prox}_{\gamma f}\mathfrak{W}&=& \mbox{arg}\!\mathop {\min}\limits_{\mathfrak{Y}\in \mathfrak{H}_0}\{\frac{L}{2}\|\mathcal{W}_{0}-\mathcal{P}\|_F^2+\sum\limits_{i=1}^{N}\frac{\lambda}{N}\|W_{i,(i)}\|_* + \frac{1}{2\gamma}\|\mathfrak{Y}-\mathfrak{W}\|^2_{\mathfrak{H}_0}\}\\
&=&\mbox{arg}\!\mathop {\min}\limits_{\mathfrak{Y}\in \mathfrak{H}_0} \{\frac{L}{2}\|\mathcal{W}_{0}-\mathcal{P}\|_F^2+\sum\limits_{i=1}^{N}\frac{\lambda}{N}\|W_{i,(i)}\|_* + \frac{1}{2(N+1)\gamma}\sum\limits_{i=0}^{N}\|\mathcal{Y}_i-\mathcal{W}_i\|_F^2\}\\
&=&(\text{prox}_{(N+1)\gamma (\frac{L}{2}\|\mathcal{W}-\mathcal{P}\|_F^2)}\mathcal{W}_0,\text{prox}_{(N+1)\gamma (\frac{\lambda}{N}\|W_{1,(1)}\|_*) }\mathcal{W}_1,\cdots,\text{prox}_{(N+1)\gamma (\frac{\lambda}{N}\|W_{N,(N)}\|_*) } \mathcal{W}_N)
\end{eqnarray*}
For $\text{prox}_{(N+1)\gamma (\frac{L}{2}\|\mathcal{W}-\mathcal{P}\|_F^2)}\mathcal{W}_0$,
we have
\begin{equation}\label{eq:F0solution}
\mbox{arg}\!\mathop {\min}\limits_{\mathcal{Y}\in \mathfrak{F}}\{\frac{L}{2}\|\mathcal{W}-\mathcal{P}\|_F^2+\frac{1}{2(N+1)\gamma}\|\mathcal{W}-\mathcal{Y}_0\|_F^2\}={(\frac{L}{2}\mathcal{P}+\frac{1}{2(N+1)\gamma}\mathcal{Y}_0)}/{(\frac{L}{2}+\frac{1}{2(N+1)\gamma})}.
\end{equation}
For $\text{prox}_{(N+1)\gamma (\frac{\lambda}{N}\|W_{i,(i)}\|_*) }\mathcal{W}_i, i=1,\cdots,N$, by Theorem 3.1 in [9], we have
\begin{equation}\label{eq:FIsolution}
\mbox{arg}\!\mathop {\min}\limits_{\mathcal{Y}\in \mathfrak{F}}\{\frac{\lambda}{N}\|W_{i,(i)}\|_*+\frac{1}{2(N+1)\gamma}\|\mathcal{W}-\mathcal{Y}_i\|_F^2\}=\text{refold}(U\mathcal{S}_{\frac{\lambda (N+1)\gamma}{N}}[S]V^T),
\end{equation}
where $USV^T$ is the SVD of $Y_{i,(i)}$, the $\text{refold}(\cdot)$ is referred to Section~\ref{sec:Notations}, and the $\mathcal{S}_{\varepsilon}[\cdot]$ is the soft-thresholding operator introduced in [19]:
\begin{equation}
\mathcal{S}_{\varepsilon}[x]\doteq \left\{ \begin{array}{l}
 x-\varepsilon, \textrm{if } x>\varepsilon,   \\
 x+\varepsilon, \textrm{if } x<-\varepsilon,   \\
0, \textrm{otherwise}   \\
 \end{array} \right.
\end{equation}
where $x\in \mathbb{R}$ and $\varepsilon > 0$. For vectors and matrices, this operator is extended by applying element-wise.

The proximal map of the indicator function $g(\mathfrak{W})$ is simply given by
\begin{eqnarray*}\label{eq:GproximalMap}
\text{prox}_{\gamma g}\mathfrak{W}&=& (\widehat{\mathfrak{W}},\cdots,\widehat{\mathfrak{W}}),
\end{eqnarray*}
where $\widehat{\mathfrak{W}}=\frac{1}{N+1}\sum\nolimits_{i=1}^{N}{\mathcal{W}_i}$.

Now apply Eq.~(\ref{eq:ConvergenceSequence}), we obtain the iteration rules for the original problem:
\begin{eqnarray}\label{eq:DRTC}
\mathcal{W}_0^{k+1}=\mathcal{W}_0^{k}+ \mbox{arg}\!\mathop {\min}\limits_{\mathcal{W}} (\frac{L}{2}\|\mathcal{W}-\mathcal{P}\|_F^2+\frac{1}{2(N+1)\gamma}\|\mathcal{W}-(2\widehat{\mathfrak{W}}-\mathcal{W}_0^{k})\|_F^2) -\widehat{\mathfrak{W}},\\
\mathcal{W}_i^{k+1}=\mathcal{W}_i^k+\mbox{arg}\!\mathop {\min}\limits_{\mathcal{W}}(\frac{\lambda}{N}\|W_{(i)}\|_*+\frac{1}{2(N+1)\gamma}\|\mathcal{W}-(2\widehat{\mathfrak{W}}-\mathcal{W}_i^k\|_F^2) -\widehat{\mathfrak{W}}, i=1,\cdots,N. 
\end{eqnarray}
The convergence is guaranteed by Theorem 4.1 in [11]. When it converges, the weight tensor is $\mathfrak{W}$.

\emph{ADM based method}: The ADM based method goes back to last century [20]. The approach consists of iteratively updating the original variables and finally carrying out the update of the dual variables. Each update involves a single variable and is conditioned to the fixed value of the others. In order to use the ADM in tensor completion, Gandy introduced $N$ new tensor-value variables that represents the $N$ different mode-$n$ unfoldings of the original tensor, then form the augmented Lagrangian and update all the variables one at a time. Following Gandy's method, we introduce $N$ new variable $\mathcal{Y}_i\in  \mathbb{R}^{I_1\times I_2\times \cdots \times I_N}$ and rephrase line 2 of the Algorithm~\ref{algo:APG} as
\begin{equation}\label{eq:ADMProblemFormulation}
\begin{array}{l}
\mathop {\text{min}}\limits_{\mathcal{W},\mathcal{Y}_i}\quad  \frac{L}{2}\|\mathcal{W}-\mathcal{P}\|_F^2+\frac{\lambda}{N} \sum\limits_{i=1}^N {\|Y_{i,(i)}\|_*} \\
\text{subject to }\quad  \mathcal{Y}_i=\mathcal{W} \quad \forall i\in \{1,\cdots,N\}.
\end{array}
\end{equation}
Let $f(\mathcal{W})=\frac{L}{2}\|\mathcal{W}-\mathcal{P}\|_F^2$, $g(\mathfrak{Y})=\frac{\lambda}{N} \sum\nolimits_{i=1}^N {\|Y_{i,(i)}\|_*}$, where $\mathfrak{Y}=(\mathcal{Y}_1,\cdots,\mathcal{Y}_N)^T$. Thus the constrain over $\mathfrak{Y}$ and $\mathcal{W}$ is $\mathfrak{Y}=(\mathcal{W},\cdots,\mathcal{W})$. Then the augmented Lagrangian of Eq.~(\ref{eq:ADMProblemFormulation}) becomes
\begin{equation}\label{eq:AugmentedLagrangian}
\mathcal{L}_A(\mathcal{W},\mathfrak{Y},\mathfrak{U})=\frac{L}{2}\|\mathcal{W}-\mathcal{P}\|_F^2+\sum\limits_{i=1}^{N}(\frac{\lambda}{N}\|Y_{i,(i)}\|_*-<\mathcal{U}_i,\mathcal{W}-\mathcal{Y}_i>+\frac{\beta}{2}\|\mathcal{W}-\mathcal{Y}_i\|_F^2)
\end{equation}
where the parameter $\beta$ is any positive number and $\mathfrak{U}=(\mathcal{U}_1,\cdots,\mathcal{U}_N)^T$ is the Lagrange multiplier. By minimization $\mathcal{L}_A(\mathcal{W},\mathfrak{Y},\mathfrak{U})$ with respect to each single variable and other variables fixed, we obtain the updating rules of all the variables $\mathfrak{Y},\mathcal{W},\mathfrak{U}$
\begin{equation}\label{eq:ADMUpdatingRules}
\left\{
\begin{array}{l}
\mathcal{W}^{k+1}=(L\mathcal{P}+\beta \sum\limits_{i=1}^{N}\mathcal{Y}_i+\sum\limits_{i=1}^{N}\mathcal{U}_i)/(L+\beta N), \\
\mathcal{Y}_i^{k+1}=\text{refold}(U\mathcal{S}_{\frac{\lambda}{\beta N}}[S]V^T), i=1,\cdots,N, \\
\mathcal{U}_i^{k+1}=\mathcal{U}_i^k-\beta (\mathcal{W}^{k+1}-\mathcal{Y}_i^{k+1}), i=1,\cdots,N,
\end{array}
\right.
\end{equation}
where $USV^T$ is the SVD of $(W_{(j)}^{k+1}-\frac{1}{\beta}U_{j,(j)}^{k})$.

Until now we have proposed two methods to solve the tensor classification problem. In the next section, we discuss the online implementation of the proposed learning processes.

\section{Online Learning}
\label{sec:OnlineImplementations}
The above proposed methods are iterative \emph{batch} procedures, accessing the whole training set at each iteration in order to minimize a weighted sum of a cost function and the tensor trace norm. This kind of learning procedure cannot deal with huge size training set for the data probably cannot be loaded into memory simultaneously, furthermore it cannot be started until the training data are prepared, hence cannot effectively deal with the training data appear in sequence, such as audio and video processing.

To address these problems, we propose an \emph{online} approach that processes the training samples, one at a time, or in mini-batches to learn the weight tensor and the bias for tensor classification. We transform the above algorithm to the online learning framework. The framework is described in Algorithm~\ref{algo:OnlineLearning} in which we also include the bias updating steps.

Our procedure is summarized in Algorithm~\ref{algo:OnlineLearning}. The $\otimes$ operator in step 6 of the algorithm denotes the Kronecker product which is similar to matrix Kronecker product. Given two tensors $\mathcal{A}\in \mathbb{R}^{I_1\cdots \times I_N}$ and $\mathcal{B}\in \mathbb{R}^{J_1\times\cdots \times J_N}$ with equal order $N$, $\mathcal{A}\otimes \mathcal{B}$ denotes the Kronecker product between $\mathcal{A}$ and $\mathcal{B}$, results as a tensor in $\mathbb{R}^{I_1J_1\times\cdots \times I_NJ_N}$, defined by blocks of sizes $J_1\times\cdots \times J_N$ equal to $a_{i_1\cdots i_N}\mathcal{B}$. $\text{GridTr}(\mathcal{W},\mathcal{B}_t)$ in step 13 denotes an operator with input $\mathcal{W}\in \mathbb{R}^{I_1\cdots \times I_N}$ and $\mathcal{B}_t\in \mathbb{R}^{I_1J_1\times\cdots \times I_NJ_N}$, result in $\mathbb{R}^{I_1\cdots \times I_N}$ with the $(i_1,\cdots,i_N)$th element defined as the inner product between $\mathcal{W}$ and the $(i_1,\cdots,i_N)$th $\mathbb{R}^{I_1\cdots \times I_N}$ block of $\mathcal{B}_t$.

Assuming the training set composed of i.i.d. samples of a distribution $p(\mathcal{X},y)$, its inner loop draws one training sample $(\mathcal{X}_t,y_t)$ at a time. This sample is first used to update the ``past'' information $\mathcal{A}_{t-1}$, $\mathcal{B}_{t-1}$, $c_{t-1}$, $\mathcal{D}_{t-1}$, $L_{t-1}$. Then the Algorithm~\ref{algo:APG} is applied to update the weight matrix with the warm start $\mathcal{W}_{t-1}$ obtained at the previous iteration. Since $F_t(\mathcal{W},b_{t-1})$ is relative close to $F_{t-1}(\mathcal{W},b_{t-1})$ for large values of $t$, so are $\mathcal{W}_t$ and $\mathcal{W}_{t-1}$, under suitable assumptions, which makes it efficient to use $\mathcal{W}_{t-1}$ as warm restart for computing $\mathcal{W}_t$.

For the stopping criteria of the inside iterations, we take the following relative error conditions:
\begin{equation}\label{eq:StoppingCriteria}
\|\mathcal{W}_{k+1,t}-\mathcal{W}_{k,t}\|_F/(\|\mathcal{W}_{k,t}\|_F+1)<\varepsilon_1\text{ and }|b_{k+1,t}-b_{k,t}|/(|b_{k,t}|+1)<\varepsilon_2.
\end{equation}

\begin{algorithm}[t]
\caption{Online learning for tensor classification via APG} \label{algo:OnlineLearning}
\textbf{Initialization} $\mathcal{W}_0=0\in  \mathbb{R}^{I_1\times I_2\times \cdots \times I_N}, b_0\in\mathbb{R}, \lambda. $

1: $\mathcal{A}_0\in \mathbb{R}^{I_1\times I_2\times \cdots \times I_N}\leftarrow 0, \mathcal{B}_0\in\mathbb{R}^{I_1I_1\times I_2I_2\times \cdots \times I_NI_N}\leftarrow 0, c_0\in \mathbb{R}\leftarrow 0, \mathcal{D}_0\in \mathbb{R}^{I_1\times I_2\times \cdots \times I_N}\leftarrow 0, L_0=0\in \mathbb{R}$ (reset the ``past'' information).

2: \textbf{for} $t=1$ \textbf{to} $T$ \textbf{do}

3: Draw training sample $(\mathcal{X}_t,y_t)$ from $p(\mathcal{X},y)$.

4: // Line 5-9 update ``past'' information.

5: $\mathcal{A}_t\leftarrow \mathcal{A}_{t-1}+y_t\mathcal{X}_t$;

6: $\mathcal{B}_t\leftarrow \mathcal{B}_{t-1}+\mathcal{X}_t\otimes \mathcal{X}_t$;

7: $c_t\leftarrow c_{t-1}+y_t$;

8: $\mathcal{D}_t\leftarrow \mathcal{D}_{t-1}+\mathcal{X}_t$;

9: $L_t\leftarrow L_{t-1}+2\prod\nolimits_{m=1}^{N} I_m\left\|\mathcal{X}_t\right\|_F^2$.

10: // Line 11-19 compute $\mathcal{W}_t$ using the APG method, with $\mathcal{W}_{t-1}$ as warm restart.

11: $\mathcal{W}_{0,t}=\mathcal{Z}_{1,t}=\mathcal{W}_{t-1}\in\mathbb{R}^{I_1\times I_2\times \cdots \times I_N},b_{0,t}=b_{t-1},\alpha_1 =1, k=1.$

12: \textbf{while} not converged \textbf{do}

13: $\mathcal{W}_{k,t}=\mbox{arg}\!\mathop {\min}\limits_{\mathcal{W}}\frac{L_t}{2}\|\mathcal{W}-(\mathcal{Z}_{k,t}+\frac{2}{L}(\mathcal{A}_t-\text{GridTr}(\mathcal{Z}_{k,t},\mathcal{B}_t)-b_{k-1,t}\mathcal{D}_t))\|_F^2+\lambda \|\mathcal{W}\|_*$.

14: $\alpha_{k+1}=\frac{1+\sqrt{1+4\alpha_k^2}}{2}$.

15: $\mathcal{Z}_{k+1,t}=\mathcal{W}_{k,t}+\frac{\alpha_{k}-1}{\alpha_{k+1}}(\mathcal{W}_{k,t}-\mathcal{W}_{k-1,t})$.

16: $b_{k,t}=\frac{1}{t}(c_t-<\mathcal{W}_{k,t},\mathcal{D}_t>)$

17: $k\leftarrow k+1$.

18: \textbf{end while}

19: $\mathcal{W}_t\leftarrow \mathcal{W}_{k,t},b_t\leftarrow b_{k,t}.$

20: \textbf{end for}

\textbf{Output}: $\mathcal{W}\leftarrow \mathcal{W}_T, b\leftarrow b_T$.
\end{algorithm}

In some conditions, use the classical heuristic in gradient descent algorithm, we may also improve the convergence speed of our algorithm by drawing $\mu > 1$ training samples at each iteration instead of a single one. Let us denote by $(\mathcal{X}_{t,1},y_{t,1}),...,(\mathcal{X}_{t,\mu},y_{t,\mu})$ the samples drawn at iteration $t$. We can now replace lines 5 and 9 of Algorithm~\ref{algo:OnlineLearning} by
\begin{equation}\label{eq:MiniBatchInfoUpdata}
\begin{array}{l}
\mathcal{A}_t\leftarrow \mathcal{A}_{t-1}+\sum\limits_{i=1}^{\mu}{y_{t,i}\mathcal{X}_{t,i}},\quad
\mathcal{B}_t\leftarrow \mathcal{B}_{t-1}+\sum\limits_{i=1}^{\mu}{\mathcal{X}_{t,i}\otimes \mathcal{X}_{t,i}}, \quad
c_t\leftarrow c_{t-1}+\sum\limits_{i=1}^{\mu}{y_{t,i}}, \\
\mathcal{D}_t\leftarrow \mathcal{D}_{t-1}+\sum\limits_{i=1}^{\mu}{\mathcal{X}_{t,i}},
\text{ and } L_t\leftarrow L_{t-1}+\sum\limits_{i=1}^{\mu}{2\prod\nolimits_{m=1}^{N} I_m\left\|\mathcal{X}_{t,i}\right\|_F^2}.
\end{array}
\end{equation}
But in real applications, this online with mini-batch update method may not improve the convergence speed on the whole since the batch past information computation (Eq.~(\ref{eq:MiniBatchInfoUpdata})) would occupy much of the time. The updating of $\mathcal{B}_t$ needs to do Kronecher product which spend much of the computing resource. If the computation cost of Eq.~(\ref{eq:MiniBatchInfoUpdata}) can be ignored or largely decreased, for example by parallel computing, this mini-batch method would increase the convergence speed by a factor of $\mu$.

\section{Experimental Validation}
In this section, we conduct experiments to demonstrate the characteristics of the proposed methods for tensor classification problem. Six algorithms are compared: the batch learning algorithm with APG using DR methods (APG\_DR); the online learning algorithm with APG using DR (OL\_APG\_DR); the batch learning algorithm with APG using ADM method (APG\_ADM); the online learning algorithm with APG using ADM (OL\_APG\_ADM); OL\_APG\_DR with update Eq.~(\ref{eq:MiniBatchInfoUpdata}) (OL\_APG\_DR\_miniBatch); OL\_APG\_ADM with update Eq.~(\ref{eq:MiniBatchInfoUpdata}) (OL\_APG\_ADM\_miniBatch). All algorithms are run in Matlab on a PC with an Intel 2.53GHz dual-core CPU and 3.25GB memory.

For our experiments, we use randomly generated $2.4\times 10^{5}$ $3$-order $10\times 10\times 10$ tensors, which are composed of varied ranks (note that here the rank is not the $n$-rank mentioned above, here the rank concept related to CANDECOMP/PARAFAC decomposition, refer [1] for exact definition); $2\times 10^{5}$ of these are kept for training, and the rest for testing. The goal is to classify the tensors according to their ranks. Hence we have made the tensor rank identification problem into a novel classification or regression formulation. We generate the rank-$r$ tensor as a sum of $r$ rank one tensors, where each rank one tensor is a outer product of 3 vectors whose elements are drawn i.i.d from the standard uniform distribution on the open interval $(0,1)$. For all the algorithm, the parameters in the stopping criteria~(\ref{eq:StoppingCriteria}) are $\varepsilon_1=10^{-10}$ and $\varepsilon_2=10^{-10}$. The regularization constant $\lambda$ is anchored by the large explicit fixed step size $L$ and the tensors involved, which means that in practice the parameter $\lambda$ should be set adaptably with the step size $L$ in the online process. But due to this variation of $\lambda$, the comparisons between the algorithms would not bring into effect. Hence in this work we use $\lambda = 1$ throughout. Considering a balance between convergence speed and accuracy, we set $\beta=10^7, \gamma=10^{-7}$ in this work.

\begin{figure}
\subfigure[Mean square rank prediction error as function of time.]{
\includegraphics[width=0.50\textwidth]{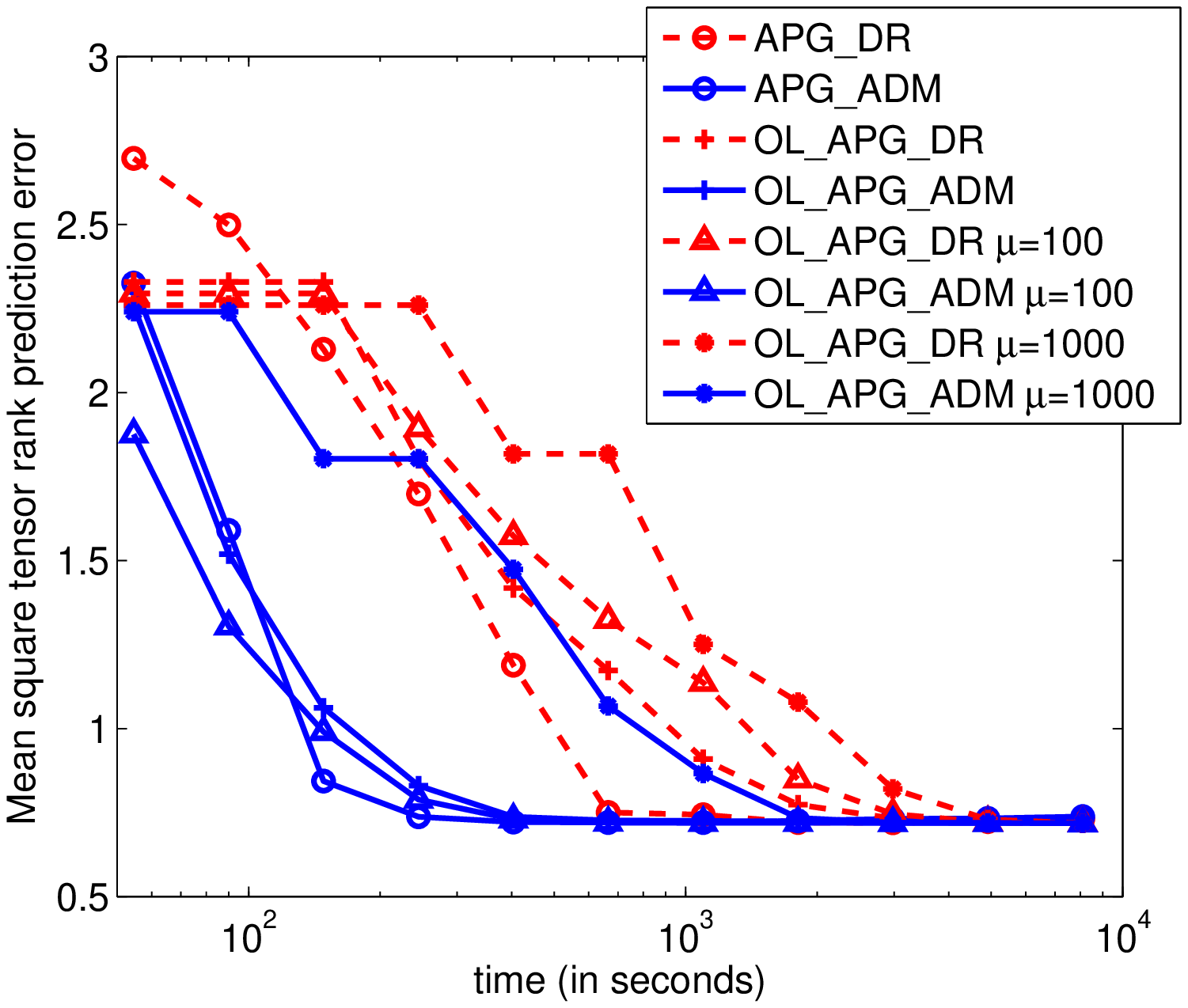}
}
\hspace{0.1in}
\subfigure[Tensor classification accuracy with $\eta=1$ as function of time.]{
\includegraphics[width=0.50\textwidth]{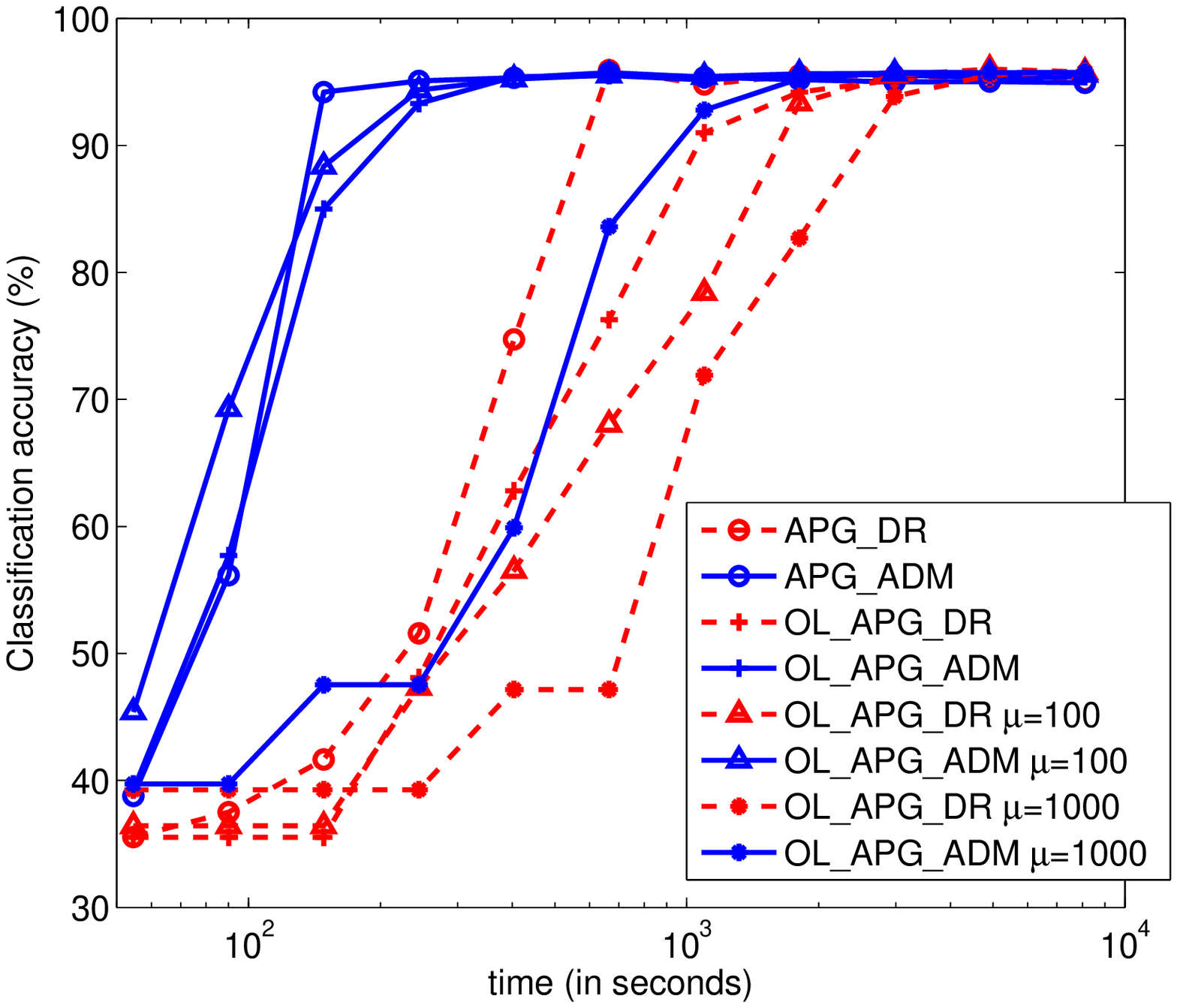}
}

\caption{Comparison between various learning methods and results are reported as functions of learning time on a logarithmic scale. }
\label{fig:PerformancCompar}       
\end{figure}

Figure~\ref{fig:PerformancCompar} compares all the algorithms proposed in this work. The batch algorithm use a training set of $2\times 10^3$ training samples, while the online algorithm draws samples from the entire training set. We use a logarithmic scale for the computation time. Figure~\ref{fig:PerformancCompar}(a) shows the mean square tensor rank prediction errors as functions of time. It can be seen generally that all methods converge. In all these methods, ADM based methods converge faster than DR based methods. The batch learning methods converge faster than corresponding online learning methods with or without mini-batch past information updating. It can also be seen that when the size of the mini-batch used in online method increase, the speed of convergence will decrease, and the reason for this has been explained in the last paragraph of Section~\ref{sec:OnlineImplementations}. After all the methods converge, they result in almost equal performance. Figure~\ref{fig:PerformancCompar}(b) shows the classification rates with tensor rank estimation error tolerances $\eta=1$. Here the rank estimation error tolerance means that if the distance between the estimation rank value and the real rank value is less than $\eta$, then the tensor classification would be right. The convergence of the classification accuracies are corresponding to the convergence of the mean square tensor rank prediction errors. With an error tolerance $\eta=1$, the methods result in a classification rate of 95.9\%.

\section{Conclusions}
\label{sec:Conclusions}
In this paper, we have proposed methods to solve tensor classification problem with a tensor trace norm regularization. We successfully employed APG method to learn parameters, during which DR and ADM are used to update weight tensor. We also give out online learning implementation for all proposed methods. In addition, for standard squared loss function, we derive the explicit form of the Lipschitz constant, which saves the computation burden in searching step size. Our empirical study on tensor classification according to tensor rank demonstrates the merits of the proposed algorithms. This is, to our knowledge, the first work on tensor norm constrained tensor classification. Some future work are worth considering, such as that the alternating between minimization with respect to weight tensor and bias may results in fluctuation of target value, thus optimization algorithm that minimization jointly on weight tensor and bias are required; for multi-classification problems with more classes, some hierarchy methods may be introduced to improve the classification accuracy.


\subsubsection*{References}

\small{
[1] Kolda, T.G. \& Bader, B.W. (2009) Tensor decompositions and applications. {\it SIAM Review}
{\bf 51}(3):455-500.

[2] Tomioka, R. \& Aihara, K. (2007) Classifying matrices with a spectral regularization.
{\it 24th International Conference on Machine Learning}, pp. 895-902.

[3] Argyriou, A. \& Evgeniou, T. \& Pontil, M. (2008) Convex multi-task feature learning. {\it Machine Learning}
{\bf 73}(3):243-272.

[4] Srebro, N. \& Rennie, J. D. M. \& Jaakkola, T. S. (2005) Maximum-margin matrix factorization.
{\it Proceedings of Advances in Neural Information Processing Systems}, pp. 1329-1336

[5] Candes, E. J. \& Recht, B. (2008) Exact matrix completion via convex optimization.
{\it Technical Report, UCLA Computational and Applied Math}.

[6] Wright, J. \& Ganesh, A. \& Rao, S. \& Peng, Y. \& Ma, Y. (2009) Robust principal component analysis:
Exact recovery of corrupted low-rank matrices via convex optimization. {\it Proceedings
of Advances in Neural Information Processing Systems}.

[7] Liu, J., Musialski, P., Wonka, P. \& Ye., J. (2009) Tensor completion for estimating missing values in
visual data. {\it IEEE 12th International Conference on Computer Vision}, pp. 2114-2121.

[8] Toh, K. \& Yun, S. (2010) An accelerated proximal gradient algorithm for nuclear norm regularized least squares
problems. {\it Pacific J. Optim} {\bf 6}:615-640.

[9] Ji, S. \& Ye, J. (2009) An accelerated gradient method for trace norm minimization.
{\it 26th International Conference on Machine Learning}, pp. 457-464.

[10] Liu, Y.J., Sun, D. \& Toh, K.C. (2009) An implementable proximal point algorithmic framework for nuclear norm minimization.
{\it Mathematical Programming}, pp. 1-38.

[11] Gandy, S., Recht, B., \& Yamada, I. (2011) Tensor completion and low-n-rank tensor recovery via convex optimization. {\it Inverse Problems} {\bf 27}(2).

[12] Bertsekas, D.P. (1999) {\it Nonlinear programming}. Athena Scientific Belmont, MA.

[13] Nesterov, Y. (1983) A method of solving a convex programming problem with convergence rate $O(\frac{1}{k_2})$.
{\it Soviet Mathematics Doklady}. {\bf 27}(2):372-376.

[14] Nesterov, Y. (2005) Smooth minimization of non-smooth functions.
{\it Mathematical Programming}. {\bf 103}(1):127-152.

[15] Douglas, J. \& Rachford, H. (1956) On the numerical solution of heat conduction problems in two and
three space variables. {\it Trans. of the American Mathematical Society} {\bf 82}:421-439.

[16] Combettes, P. L. \& Pesquet, J. C. (2007) A Douglas-Rachford splitting approach to nonsmooth
convex variational signal recovery. {\it IEEE J. Sel. Top. Signal Process} {\bf 1}(4):564-574.

[17] Moreau, J.J. (1962) Fonctions convexes duales et points proximaux dans un espace hilbertien.
{\it C.R.Acad.Sci. Paris Ser. A Math} {\bf 244}:2897-2899.

[18] Combettes, P. L. \& Wajs, V.R. (2005) Signal recovery by proximal forward-backward splitting. {\it SIAM
Multiscale Model. Simul.} {\bf 4}:1168-1200.

[19] Lin, Z., Chen, M., Wu, L. \& Ma, Y. (2009) The augmented lagrange multiplier method for exact recovery
of corrupted low-rank matrices. preprint.

[20] Gabay, D. \& Mercier, B. (1976) A dual algorithm for the solution of nonlinear variational problems
via finite-element approximations. {\it Comp. Math. Appl.} {\bf 2}:17-40.

\end{document}